\documentclass[11pt]{article}
\usepackage{amsfonts,amsmath,amsthm}
\usepackage{array}
\usepackage{amscd}
\usepackage{enumerate}
\usepackage{color}

\newtheorem{definition}{Definition}
\newtheorem{theorem}{Theorem}
\newtheorem{lemma}{Lemma}
\newtheorem{corollary}{Corollary}

\newcounter{ex}
\newenvironment{example}{\par\noindent\textbf{Example \stepcounter{ex}\arabic{ex}.}}{\bigskip}

\newcommand{\bN}{{\mathbb{N}}}

\begin{document}
\title{\bf Almost Oscillation Criteria for Second Order Neutral Difference Equation with Quasidifferences} 
\date{}

\author{R. Jankowski\footnote{University of Bia\l ystok, Institute of Mathematics, Bia\l ystok, Poland, email: rjjankowski@math.uwb.edu.pl}
\and
E. Schmeidel\footnote{University of Bia\l ystok, Institute of Mathematics, Bia\l ystok, Poland, email: eschmeidel@math.uwb.edu.pl} 
}

\maketitle
\begin{abstract}
\noindent
Using the Riccati transformation techniques, we will extend some almost oscillation criteria for the second-order nonlinear neutral difference equation with quasidifferences
\[
\Delta\left(r_n\left(\Delta \left(x_n+c x_{n-k}\right)\right)^{\gamma}\right)+q_nx_{n+1}^{\alpha}=e_n.
\]
\end{abstract}

\section{Introduction}
Recently there has been an increasing interest in the study of the qualitative behavior
of solutions of neutral difference equations (see the monographs \cite{bib1B}--\cite{AW}, \cite{bib3B}). Particularly, the oscillation and nonoscillation of solutions of the second-order neutral difference equations attract attention; see the papers \cite{DGJ2002}, \cite{bibGL}, \cite{bibLG}--\cite{bibLB}, \cite{bibmm2}, \cite{S1}-\cite{bib6} and the references therein. The interesting oscillatory results for first order and even order neutral difference equations can be found in \cite{MM2004} and \cite{MM2009}.

In the presented paper, the results obtained in \cite{TVG} by Thandapani, Vijaya and Gy\H{o}ri for 
\begin{equation*}
\Delta^2 \left(x_n+c x_{n-k}\right)^{\gamma}+q_nx_{n+1}^{\alpha}=e_n,
\end{equation*}
are generalized to the second order nonlinear neutral difference equation with quasidifference of the form  
\begin{equation}\label{e0}
\Delta\left(r_n\left(\Delta \left(x_n+c x_{n-k}\right)\right)^{\gamma}\right)+q_nx_{n+1}^{\alpha}=e_n.
\end{equation}
Here $k$ is a nonnegative integer, $\Delta$ is the forward difference operator defined by $\Delta x_n=x_{n+1}-x_n$, $c$ is a real nonnegative constant, $\alpha>\gamma\geq 1$ are ratios of odd positive integers, $\left(r_n\right)$ and $\left(q_n\right)$ and $\left(e_n\right)$ are positive sequences defined on $\bN=\{1,2,3,\ldots\}$. 

Finally, as a corollary of our main result, almost oscillation property of solutions of a special case of equation \eqref{e0} in the form
\begin{equation}\label{SL}
\Delta\left(r_n\Delta x_n\right)+q_nx_{n+1}^\alpha=e_n
\end{equation}
is studied. For $\alpha =1$, equation \eqref{SL} is known as the forced second order Sturm-Liouville difference equation. Some oscillation results for equation \eqref{SL} were investigated among others by Do\u{s}l\'{y}, Graef and Jaro\u{s} in \cite{DGJ2002}. 

By a solution of equation~\eqref{e0} we mean a real valued sequence $\left(x_n\right)$ defined on $\bN_k:=\{k,k+1,\ldots\}$, which satisfies~\eqref{e0} for every $n\in\bN_k$. %and $x_s=\varphi_s$ for $s=k,k+1,\ldots,1$, where $\varphi_s$ is the given real initial sequence. 

Sequence $\left(x_n\right)$ is said to be oscillatory, if for every integer $n_k\in \bN_k$, there exists $n\geq n_k$ such that $x_nx_{n+1}\leq 0$; otherwise, it is called nonoscillatory.

\begin{definition}
Solution $\left(x_n\right)$ of equation~\eqref{e0} is said to be almost oscillatory if either $\left(x_n\right)$ is oscillatory, or $\left(\Delta x_n\right)$ is oscillatory, or $x_n\rightarrow 0$ as $n\rightarrow \infty$.
\end{definition}

We begin with some lemmas which will be used for proving the main result.
\begin{lemma}\label{l1}
Set 
\begin{equation}\label{Fx}
F\left(x\right)=a \ x^{\alpha-\gamma}+ \frac{b}{x^\gamma} \mbox{ for } x>0.
\end{equation}
If $a\geq 0, b\geq 0$ and
\begin{equation}\label{1}
\alpha>\gamma\geq 1
\end{equation}
then $F\left(x\right)$ attains its minimum
\begin{equation*}
F_{min}=\frac{\alpha a^{\frac{\gamma}{\alpha}}b^{1-\frac{\gamma}{\alpha}}}{\gamma^{\frac{\gamma}{\alpha}}\left(\alpha-\gamma\right)^{1-\frac{\gamma}{\alpha}}}.
\end{equation*}
\end{lemma}

\begin{lemma}\label{l2}
For all $x\geq y$ and $\gamma \geq 1$ we have the following inequality

$$  x^\gamma-y^\gamma\geq \left(x-y\right)^\gamma.$$

\end{lemma}

\section{Almost Oscillation Criterion}

In this section, by using the Riccati substitution we will establish new almost oscillation
criterion for equation \eqref{e0}.

\begin{theorem}\label{t1}
Let 
\begin{equation}\label{P}
\left(r_n\right), \left(q_n\right) \mbox{ and } \left(e_n\right) \mbox{  be positive sequences,}
\end{equation}
sequence $(r_n)$ is bounded, it means that there exists positive real constant $R$ such that
\begin{equation}\label{R}
r_n \leq R \mbox{ for } n \in \bN.
\end{equation}
Assume also that 
\begin{equation}\label{AG}
\alpha>\gamma\geq 1 \mbox{ are ratios of odd positive integers.}
\end{equation} 
If there exist positive sequences $(p_n)$ such that
\begin{equation}\label{wzor 1}
\limsup_{n\rightarrow \infty} \sum_{i=1}^n\left(p_iQ_i-\frac{R\left(\Delta p_i\right)^2}{4p_i}\right)=\infty 
\end{equation}
and
\begin{equation}\label{wzor 2}
\sum_{i=1}^\infty \sum_{j=1}^{i-1} \left(M q_j\pm e_j\right)^{\frac{1}{\gamma}}=\infty,
\end{equation}
where $$Q_n^{*}=\frac{d^{\alpha-\gamma}q_n}{\left(1+c\right)^\alpha}-d^{-\gamma}e_n,$$ 
$$Q^{**}_n = \frac{\alpha q_n^{\frac{\gamma}{\alpha}}e_n^{1-\frac{\gamma}{\alpha}}}{\gamma^{\frac{\gamma}{\alpha}}\left(\alpha-\gamma\right)^{1-\frac{\gamma}{\alpha}}\left(1+c\right)^{\gamma}},$$
\begin{equation}\label{Qn}
Q_n = \min \lbrace{ Q_n^{*}, Q^{**}_n \rbrace},
\end{equation}
(here $d>0$ and $M>0$ are suitable constants), then every solution of equation \eqref{e0} is almost oscillatory.
\end{theorem}

\begin{proof}
Set
\begin{equation}\label{z_n}
z_n=x_n+c_nx_{n-k}
\end{equation}  
then equation  \eqref{e0} takes the following form
$$ \Delta\left(r_n\left(\Delta z_n\right)^\gamma\right)=-q_nx_{n+1}^{\alpha}+e_n. $$

Suppose, for the contrary,  that sequence $\left(x_n\right)$ is a solution eventually of one sign of equation \eqref{e0} such that $(\Delta x_n)$ is eventually of one sign as well. \\

Assume first that $\left(x_n\right)$ is an eventually positive sequence.
It means that there exists $n_0 \in \bN$ such that $x_{n-k}>0$ for all $n\geq n_0$. 
We have two possibilities to consider: \\
$\left(Ia\right)\ \ \Delta x_n>0$ eventually or \\
$\left(Ib\right)\ \  \Delta x_n<0$ eventually.
 
Case $\left(Ia\right)$: Assume that  $\Delta x_n>0$. Then $\Delta z_n>0$. We have $x_n\geq \frac{z_n}{1+c}$, then from equation \eqref{e0}, we get
\begin{equation}\label{equation}
\Delta\left(r_n\left(\Delta z_n\right)^\gamma\right)\leq \frac{-q_n}{\left(1+c\right)^\alpha}z_{n+1}^{\alpha}+e_n.
\end{equation} 
Let us denote by $(w_n)$ the following sequence 
 \begin{equation}\label{w}
 w_n\colon=p_n\frac{r_n\left(\Delta z_n\right)^\gamma}{z_{n+1}^\gamma},
 \end{equation}
where $z_n$ is defined by \eqref{z_n}. Then $w_n>0$ for $n\geq n_0$. We have
\begin{align*}
\Delta w_n &= p_{n+1}\frac{r_{n+1}\left(\Delta z_{n+1}\right)^\gamma}{z_{n+2}^\gamma}-p_n\frac{r_n\left(\Delta z_n\right)^\gamma}{z_{n+1}^\gamma}=\\
& =p_n\frac{\Delta\left(r_n\left(\Delta z_n\right)^\gamma\right)}{z_{n+1}^\gamma}+  p_{n+1}\frac{r_{n+1}\left(\Delta z_{n+1}\right)^\gamma}{z_{n+2}^\gamma}- p_{n}\frac{r_{n+1}\left(\Delta z_{n+1}\right)^\gamma}{z_{n+2}^\gamma}+\\
&+p_{n}\frac{r_{n+1}\left(\Delta z_{n+1}\right)^\gamma}{z_{n+2}^\gamma}-p_{n}\frac{r_{n+1}\left(\Delta z_{n+1}\right)^\gamma}{z_{n+1}^\gamma}\\
 &=p_n\frac{\Delta\left(r_n\left(\Delta z_n\right)^\gamma\right)}{z_{n+1}^\gamma}+\frac{\Delta p_n}{p_{n+1}}w_{n+1}+\frac{p_n r_{n+1}\left(\Delta z_{n+1}\right)^\gamma}{z_{n+2}^\gamma z_{n+1}^\gamma}\left[z_{n+1}^\gamma-z_{n+2}^\gamma\right]
\end{align*}
and finally
\begin{equation}\label{deltawn}
\Delta w_n =p_n\frac{\Delta r_n\left(\Delta z_n\right)^\gamma}{z_{n+1}^\gamma}+ \frac{\Delta p_n}{p_{n+1}}w_{n+1}-\frac{ p_n}{p_{n+1}}w_{n+1}\frac{\Delta z_{n+1}^\gamma}{z_{n+1}^\gamma}.
\end{equation}
From the above and \eqref{equation}, we get
\begin{equation}\label{w_n}
 \Delta w_n\leq -p_n\left(\frac{q_n}{\left(1+c\right)^\alpha}z_{n+1}^{\alpha-\gamma}-\frac{e_n}{z_{n+1}^\gamma}\right)+\frac{\Delta p_n}{p_{n+1}}w_{n+1}-\frac{ p_n}{p_{n+1}}w_{n+1}\frac{\Delta z_{n+1}^\gamma}{z_{n+1}^\gamma}.
\end{equation}
Let $G\left(x\right)=\frac{q_n}{\left(1+c\right)^\alpha}x^{\alpha-\gamma}-\frac{e_n}{x^\gamma}$. It is easy to verify that function $G$ is increasing for positive arguments. Since $x$ is increasing, there is a constant $d>0$ such that $x\geq d >0 $ and
\begin{equation}\label{I1}
G\left(x\right)\geq \frac{q_n}{\left(1+c\right)^\alpha}d^{\alpha-\gamma}-e_nd^{-\gamma}\colon = Q_n^*.
\end{equation}
From $\eqref{w_n}$ and $\eqref{I1}$, we get the following inequality

\begin{equation*}
 \Delta w_n 
  \leq -p_n Q_n^*+ \frac{\Delta p_n}{p_{n+1}}w_{n+1}-\frac{ p_n}{p_{n+1}}w_{n+1}\frac{\left(\Delta z_{n+1}\right)^\gamma}{z_{n+2}^\gamma}.
\end{equation*}
For $\Delta z_n >0$ we have $z_{n+2} > z_{n+1}$ and $z_{n+2}^\gamma > z_{n+1}^\gamma$. Because of positivity of the sequence $(z_n)$ for large $n$, say $n \geq n_1 \geq n_0$, we obtain $\frac{1}{z_{n+2}^\gamma} < \frac{1}{ z_{n+1}^\gamma}$ for $n \geq n_1$. From \eqref{w} and by Lemma \ref{l2}, we get
\begin{equation}\label{w_n3}
\Delta w_n \leq  -p_n Q_n^*+ \frac{\Delta p_n}{p_{n+1}}w_{n+1}-\frac{ p_n}{p^2_{n+1}r_{n+1}}w^2_{n+1}
\end{equation}
This and \eqref{R} imply that
\begin{align*}
 \Delta w_n &
  \leq -p_n Q_n^*+ \frac{\left(\Delta p_n\right)^2 r_{n+1}}{4p_{n}}-\left[\sqrt{\frac{ p_n}{r_{n+1}}}\frac{1}{p_{n+1}}w_{n+1}-\frac{\sqrt{r_{n+1}}\Delta p_n}{2\sqrt{p_n}}\right]^2\\
  & \leq -p_n Q_n^*+ \frac{\left(\Delta p_n\right)^2 r_{n+1}}{4p_{n}}\leq-\left[p_n Q_n^*- \frac{\left(\Delta p_n\right)^2 R}{4p_{n}}\right].\\
\end{align*}
Summing both sides of the above inequality from $i=n_1$ to $n-1$, we obtain 
\[
w_n-w_{n_1} < -\sum_{i=n_1}^{n-1} \left(p_i Q_i^*-\frac{R(\Delta p_i)^2}{4p_{i}}\right).
\]
From positivity of $(w_n)$, we have $w_{n_1} > w_{n_1}-w_n$. Hence
\[
w_{n_1}>\sum_{i=n_1}^{n-1} \left(p_i Q_i^*-\frac{R(\Delta p_i)^2}{4p_{i}}\right).
\]
Letting $n$ into infinity we obtain
\[
w_{n_1}> \limsup_{n\rightarrow \infty} \sum_{i=n_1}^{n-1} \left(p_i Q_i^*-\frac{R(\Delta p_i)^2}{4p_{i}}\right).
\]
From \eqref{Qn}, we get
\[
w_{n_1}> \limsup_{n\rightarrow \infty} \sum_{i=n_1}^{n-1} \left(p_i Q_i-\frac{R(\Delta p_i)^2}{4p_{i}}\right).
\]
This is a contradiction with \eqref{wzor 1}.\\
%
%
%
%dla $x_n<0$ transformacja $y_n=-x_n$ daje identyczny wynik jak w tej publikacji.
%
%
Case $\left(Ib\right)$: If $\Delta x_n<0$ then $\Delta z_n<0$. From $x_n>0$ and $\Delta x_n<0$ we get $\lim \limits_{n\rightarrow \infty}x_n=l>0$. Then $x_{n+1}^\alpha\rightarrow l^\alpha>0$ as $n\rightarrow \infty$. Hence, there exists $n_2 \in \bN$ such that $x_{n+1}^\alpha \geq l^\alpha$ for $n\geq n_2$. Therefore,
we have
$$ \Delta\left(r_n\left(\Delta z_n\right)^\gamma\right)\leq -q_n l^\alpha+e_n. $$
Set $l^\alpha = M$. Summing the last inequality from $n_2$ to $n-1$, we obtain
 $$ r_n\left(\Delta z_n\right)^\gamma < r_n\left(\Delta z_n\right)^\gamma  - r_{n_2}\left(\Delta z_{n_2}\right)^\gamma  \leq -\left(\sum_{i=n_2}^{n-1}M q_i-e_i\right) $$
and 
 $$ \Delta z_n\leq -\left(\sum_{i=n_2}^{n-1}M q_i-e_i\right)^{\frac{1}{\gamma}}r_n^{-{\frac{1}{\gamma}}} , \textrm{ for } n\geq n_2. $$
Summing again the above  inequality from $n_2$ to $n$, we obtain
 $$ z_{n+1} \leq z_{n_2} -\sum_{i=n_2}^{n}\left(\sum_{j=n_2}^{i-1}M q_j-e_j\right)^{\frac{1}{\gamma}}r_i^{-{\frac{1}{\gamma}}}.$$ 
From \eqref{R}, we get 
$$z_{n+1} \leq z_{n_2} -R^{-{\frac{1}{\gamma}}}\sum_{i=n_2}^{n}\left(\sum_{j=n_2}^{i-1}M q_j-e_j\right)^{\frac{1}{\gamma}}. $$
Letting $n$ into $\infty$, from condition \eqref{wzor 2} we obtain that the right side of the above inequality is negative. So, $z_n$ is eventually negative, too. This contradiction ended the proof in this case.
\\

Finally, we assume that $\left(x_n\right)$ is an eventually negative sequence. It means that there exists $n_3 \in \bN$ such that $x_n<0$ for all $n\geq n_3$. We use the transformation $y_n = -x_n$ in the equation \eqref{e0}. Equation \eqref{e0} takes the following form 
 \begin{equation}\label{e00}
\Delta\left(r_n\left(\Delta \left(y_n+cy_{n-k}\right)\right)^{\gamma}\right)+q_ny_{n+1}^{\alpha} =- e_n.
 \end{equation}
Here sequence $(y_n)$ is an eventually positive solution of equation \eqref{e00}.
\\
We have two possibilities to consider: \\
$\left(IIa \right)\ \ \Delta y_n>0$ eventually or\\
$\left(IIb\right)\ \  \Delta y_n<0$ eventually. \\
Case $\left(IIa\right)$: Assume that  $\Delta y_n>0$. \\
From \eqref{deltawn}, by \eqref{e00},  we have
 \begin{equation*}
 \Delta w_n\leq -p_n\left(\frac{q_n}{\left(1+c\right)^\alpha}z_{n+1}^{\alpha-\gamma}+\frac{e_n}{z_{n+1}^\gamma}\right)+\frac{\Delta p_n}{p_{n+1}}w_{n+1}-\frac{ p_n}{p_{n+1}}w_{n+1}\frac{\Delta z_{n+1}^\gamma}{z_{n+1}^\gamma}.
\end{equation*}
Putting $a= \frac{q_n}{\left(1+c\right)^\alpha}$, $b=e_n$ and $x=z_{n+1}$ in \eqref{Fx}, we have 
\[
F(z_{n+1}) = \frac{q_n}{\left(1+c_{n+1}\right)^\alpha}z_{n+1}^{\alpha - \gamma}+\frac{e_n}{z_{n+1}^{\gamma}}.
\]
By Lemma \ref{l1}, we get 
\[%begin{equation}\label{lem2}
F(z_n) \geq \frac{\alpha q_n^{\frac{\gamma}{\alpha}}e_n^{1-\frac{\gamma}{\alpha}}}{\gamma^{\frac{\gamma}{\alpha}}\left(\alpha-\gamma\right)^{1-\frac{\gamma}{\alpha}}\left(1+c_{n+1}\right)^\gamma}=Q_n^{**}. 
\]%end{equation}
and 
\[
\Delta w_n \leq  -p_n Q_n^{**}+ \frac{\Delta p_n}{p_{n+1}}w_{n+1}-\frac{ p_n}{p^2_{n+1}r_{n+1}}w^2_{n+1}
\]
is satisfied. The rest of the proof is similar to proof of case $\left(Ia\right)$ and hence is omitted.\\
Case $\left(IIb\right)$: Assume that  $\Delta y_n<0$. Hence sequence $y_n$ has positive limit and the proof of this case is similar to case $\left(Ib\right)$ and hence is omitted.\\
 The proof is now complete.
\end{proof}
We illustrate the Theorem \ref{t1} by the following examples.

\begin{example}
Let us consider the difference equation

$$
\Delta\left((2-\frac{(-1)^n}{n})\left(\Delta \left(x_n+\frac{1}{2} x_{n-1}\right)\right)^{3}\right)+4 x_{n+1}^{5}=\frac{1}{n(n+1)}.
$$
Here $r_n =2-\frac{(-1)^n}{n}$, $c=\frac{1}{2}$, $\gamma=3$, $q_n = 4$, $\alpha = 5$ and $e_n = \frac{1}{n(n+1)}$. For $p_n=1$, all assumptions of Theorem \ref{t1} are satisfied. Hence, any solution of the above equation is almost oscillatory. Sequence $x_n=(-1)^{n+1}$ is one of such solutions. Here, $\left(x_n\right)$ is oscillatory. 
\end{example}

\begin{example}
Let us consider the difference equation

$$
\Delta\left((2+(-1)^n)\left(\Delta \left(x_n+2 x_{n-2}\right)\right)\right)+ x_{n+1}^{3}=14+11(-1)^{n+1}.
$$
Here $r_n =2+(-1)^n$, $c=2$, $\gamma=1$, $q_n =1$, $\alpha = 3$ and $e_n = 14+11(-1)^{n+1}$. For $p_n=1$, all assumptions of Theorem \ref{t1} are satisfied. Hence, any solution of the above equation is almost oscillatory. Sequence $x_n=2+(-1)^{n+1}$ is one of such solutions. Here, $\left(x_n\right)$ is nonoscillatory but $(\Delta x_n)$ oscillates. 
\end{example}

\begin{example}
Let us consider the difference equation

$$
\Delta\left(\frac{1}{3n+4}\left(\Delta \left(x_n+2 x_{n-1}\right)\right)\right)+n\left(n+2\right)^2  x_{n+1}^{3}=\frac{3+n^2\left(n+1\right)\left(n+3\right)}{n\left(n+1\right)\left(n+2\right)\left(n+3\right)}.
$$
Here $r_n =\frac{1}{3n+4}$, $c=2$, $k=1$, $\gamma=1$, $q_n = n\left(n+2\right)^2$, $\alpha = 3$ and 
\[
e_n =\frac{3+n^2\left(n+1\right)\left(n+3\right)}{n\left(n+1\right)\left(n+2\right)\left(n+3\right)}.
\]
For $p_n=1$, all assumptions of Theorem \ref{t1} are satisfied. Hence, any solution of the above equation is almost oscillatory. In fact, sequence $x_n=\frac{1}{n+1}$ is one of such solutions. Here, $\left(x_n\right)$ tends to zero.  
\end{example}

Assuming that $c = 0$ and $\gamma = 1$ equation \eqref{e0} takes the form \eqref{SL}. 
\begin{corollary}
Assume that conditions \eqref{P}, \eqref{R} and \eqref{AG} are held. If there exist positive sequences $(p_n)$ such that 
\begin{equation*}
\limsup_{n\rightarrow \infty} \sum_{i=1}^n\left(p_iQ_i-\frac{R\left(\Delta p_i\right)^2}{4p_i}\right)=\infty 
\end{equation*}
and
\begin{equation*}
\sum_{i=1}^\infty \sum_{j=1}^{i-1} \left(M q_j\pm e_j\right)=\infty
\end{equation*}
where $$Q_n^{*}=\frac{d^\alpha q_n-e_n}{d},$$ 
$$Q^{**}_n = \frac{\alpha q_n^{\frac{1}{\alpha}}e_n^{1-\frac{1}{\alpha}}}{\left(\alpha-1\right)^{1-\frac{1}{\alpha}}},$$
and
\begin{equation*}
Q_n = \min \lbrace{ Q_n^{*}, Q^{**}_n \rbrace},
\end{equation*}
(here $d>0$ and $M>0$ are suitable constants), then every solution of equation \eqref{SL} is almost oscillatory.
\end{corollary}

\end{document}